\documentclass[a4 paper, 12pt]{article}

\usepackage{amsmath,mathrsfs,pictexwd,dcpic}
\usepackage{amssymb}
\usepackage[left = 3cm,right=2cm,top=3cm]{geometry}
\usepackage{amsthm}
\usepackage{caption}
\usepackage{graphicx}
\usepackage{newlfont,amsfonts}
\usepackage{graphics}
\usepackage{ytableau}
\usepackage{enumitem}
\theoremstyle{definition}
\newtheorem{thm}{Theorem}[section]
\newtheorem{lem}[thm]{Lemma}
\newtheorem{Def}[thm]{Definition}

\newtheorem{rem}{Remark}
\newtheorem{ex}[thm]{Example}

\newtheorem*{Not}{Notation}
\newtheorem*{note}{Note}

\allowdisplaybreaks
\title{Study of $\textbf{\textit{\lowercase{p}}}$-Young tableaux, Robinson-Schensted correspondence and the lacunary Cauchy identity of group algebras $KG_{\lowercase{r}}$ and $KSG_{\lowercase{r}}$}
\author{M. Parvathi, A. Tamilselvi and D. Hepsi\\
	Ramanujan Institute for Advanced Study in Mathematics,\\ University of Madras, Chennai, INDIA.\\
	E.mail: tamilselvi.riasm@gmail.com}
\date{}
\begin{document}
	\maketitle
	\begin{abstract}
		In this paper, we develop the Robinson-Schensted correspondence between the elements of the groups $G_{r}$ $(\mathbb{Z}_{p^{r}}\rtimes \mathbb{Z}^{*}_{p^{r}})$ and $SG_{r}$ $(\mathbb{Z}_{p^{r-1}}\rtimes \mathbb{Z}^{*}_{p^{r}})$, along with a pair of the standard \textbf{\textit{p}}-Young tableaux. This approach differs from the classical method, and ours is based on matrix units arising from orthogonal primitive idempotents computed for every group algebra. Some classical properties of the Robinson-Schensted correspondence are discussed. As a by-product, we also extend the Cauchy identity to our setup, which we refer to as the lacunary Cauchy identity. This study offers new insights into the representation theory of these groups and their combinatorial structures.
		
		
		
	\end{abstract}
	
	\section{Introduction}
	The Robinson-Schensted correspondence is a fundamental result of the theory of symmetric functions and the representation theory. This establishes a bijection between permutations and pairs of standard Young tableaux of the same shape \cite{[BS]}. This correspondence was originally discovered by Craige Schensted in 1960 and later extended by Donald Knuth. The extension by Knuth, known as the Robinson-Schensted-Knuth correspondence, is a bijection between nonnegative matrices and the pair of semistandard Young tableaux that has broadened its applications to symmetric functions and random matrices \cite{[RS]}. A different form of the correspondence was described by Robinson in 1938 \cite{[R]}.
	
	Andrei Zelevinsky further generalized it to include ``pictures", expanding its relevance to complex geometric and algebraic structures \cite{[AVZ]}. In addition, it was extended in the following papers \cite{[PT1]}, \cite{[PT2]}, \cite{[TD]}, \cite{[TVK]}, and \cite{[HL]}. In general, this correspondence serves as a crucial link between various mathematical concepts and contributes significantly to the understanding of symmetric functions and related fields.
	
	In Section 2, we discuss some basic definitions and results, as well as review the foundational concepts relevant to this study. Section 3 focuses on establishing a one-to-one correspondence between the group $G_{r}$ (and $SG_{r}$) and the pairs of standard \textbf{\textit{p}}-Young tableaux  ($\textbf{\textit{SYT}}_{\textbf{\textit{p}}}$), illuminating the key properties of this bijection. Our approach differs from previous methods in that we establish matrix units for every group representation of $G_{r}$ and $SG_{r}$, leveraging them to naturally establish a correspondence between the elements and pairs of standard \textbf{\textit{p}}-Young tableaux.
	
	Section 4 introduces the notion of semi-standard \textbf{\textit{p}}-Young tableaux $(\textbf{\textit{SSYT}}_{\textbf{\textit{p}}})$, along with the concepts of \textbf{\textit{p}}-content and \textbf{\textit{p}}-Schur functions, culminating in the exploration of lacunary Cauchy identity. 
	
		\section{Preliminaries}
	\begin{Not}	
	
	The hook partition $\lambda(k,i)$ denotes the shape of a Young diagram with $k$ boxes, where the first row contains $k-i$ boxes, and the first column (excluding the first row) contains $i$ boxes.

	\begin{figure}[!ht]
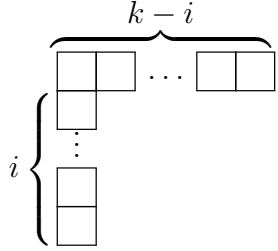

		\hspace*{6.8cm}$\overbrace{\hspace{7em}}^{\displaystyle k-i}$
		
		\centering
		\ytableausetup{centertableaux,boxsize=1.2em}
		\begin{tabular}{r@{}l}
			\raisebox{-1.5ex}{$i\left\{\vphantom{\begin{array}{c}~\\[6ex] ~
				\end{array}}\right.$} &
			\begin{ytableau}
				~               &           & \none[\hspace{0.3cm}\dots]  \hspace*{0.3cm} & &    \\ 
				~               & \none & \none      \\ 
				\none[\vdots]   & \none & \none      \\
				~               & \none & \none      \\
				~               & \none   
			\end{ytableau}
			\\[-0.5ex]
		\end{tabular}
		\caption{Diagram of $\lambda(k,i) = [k-i,1^{i}]$}
	\end{figure}
\end{Not}
	
	\noindent\textbf{Recall \cite{[TH]}:}
	\begin{itemize}
		\item In our study, we consider group algebras over a field of characteristic zero, which encompasses the $p^{r}$-th root of unity for all $r\geq 0$ and $p$ is an odd prime.
		\item Let $G_{r}$ be the group such that
		$$G_{r}=\{\tau^{i}g^{j}|~1\leq i\leq p^{r-1}(p-1),~ 1\leq j \leq p^{r} \}$$ where $g^{p^{r}}=e,~\tau^{p^{r-1}(p-1)}=e$ and $G_{r} \cong \mathbb{Z}_{p^{r}}\rtimes\mathbb{Z}^{*}_{p^{r}}$.
		\item  Let $SG_{r}$ be the subgroup of $G_{r}$ such that  $$SG_{r}=\{\tau^{i}g^{pj}|~1\leq i\leq p^{r-1}(p-1),~ 1\leq j \leq p^{r-1} \}$$
		where $(g^{p})^{p^{r-1}}=e,~\tau^{p^{r-1}(p-1)}=e$ and $SG_{r} \cong \mathbb{Z}_{p^{r-1}}\rtimes\mathbb{Z}^{*}_{p^{r}}$.     
		\item In \cite{[TH]}, we define the irreducible representations of the groups $G_{r}$ and $SG_{r}$. We also establish a bijection between the irreducible representations and the hook partitions. Using this bijection, each irreducible representation is indexed by a hook partition as follows: 
		
		Indexing set of the complete set of inequivalent irreducible representations of $G_{r}$
		\begin{itemize}
			\item $V^{r}_{0}=\{\lambda(p^{r-1}(p-1),i)\}_{0\leq i<p^{r-1}(p-1)}$, where each element in this set represents a one-dimensional representation.
			\item For $1\leq s\leq r$, $V_{s}^{r}=\{\lambda(p^{r-s}[2sp-(2s+1)],i)\}_{p^{r-s}[sp-(s+1)]\leq i<sp^{r-s}(p-1)}$, where each element in this set represents a $p^{s-1}(p-1)$-dimensional representation.
		\end{itemize}
		\item Indexing set of the complete set of inequivalent irreducible representations of $SG_{r}$
		\begin{itemize}
			\item $W^{r}_{0}=\{\lambda(p^{r-1}(p-1),i)\}_{0\leq i<p^{r-1}(p-1)}$, where each element in this set represents a one-dimensional representation.
			\item For $1\leq s\leq r-1$, $$W_{s}^{r}=\{\lambda(p^{r-s-1}[(2s+1)p-2(s+1)],i)\}_{p^{r-s-1}[sp-(s+1)]\leq i<(s+1)p^{r-s-1}(p-1)},$$ where each element in this set represents a $p^{s-1}(p-1)$-dimensional representation. 
		\end{itemize}
		\item $	E^{(r,s)}_{\rho^{i}} = \dfrac{1}{p^{r}}\sum\limits_{j=1}^{p^{r}}\rho^{ij}g^{j}$,	where $\rho$ is a primitive $p^{s}$-th root of unity, $(i,p^{s})=1$, $0\leq s\leq r$ and $r\geq 1$. 
		\item $\mathbb{E}^{(r,s)}_{\rho^{i}} = \dfrac{1}{p^{r-1}}\sum\limits_{j=1}^{p^{r-1}}\rho^{ij}(g^{p})^{j}$,
		where $\rho$ is a primitive $p^{s}$-th root of unity, $(i,p^{s})=1$ and $\ 0\leq s \leq r-1$ and $r\geq 1$.
		\item $T^{(r,s)}_{l} = \dfrac{1}{p^{r-s}}\sum\limits_{j=1}^{p^{r-s}}\beta^{lj}\tau^{jp^{s-1}(p-1)}, $	where $\beta$ is a $p^{r-s}$-th root of unity, $0\leq l \leq p^{r-s}-1$, $0\leq s\leq r$, and $r\geq 1$.
		\item $T^{(r)}_{l} = \dfrac{1}{p^{r-1}(p-1)}\displaystyle\sum_{j=1}^{p^{r-1}(p-1)}\zeta^{lj}\tau^{j},~0\leq l < p^{r-1}(p-1)$, where $\zeta$ is a $p^{r-1}(p-1)$-th root of unity and  $r\geq 1$.

	\end{itemize}
	\begin{rem}
		The block $B_{m,n}$ contains $m$-nodes horizontally and $n$-nodes vertically.
	\end{rem}
	\begin{Def}\label{block}\cite{[TH]}
		A hook partition $\lambda$ is obtained from a hook partition $\mu$ by adding a block $B_{m,n}$ of size $m+n$. This addition involves adding $m$ nodes at the end of the first row and $n$ nodes at the end of the first column. Similarly, removing block $B_{m,n}$ from the partition involves removing the last $m$ nodes in the first row and the last $n$ nodes in the first column. 
		
		\begin{figure}[!ht]
			\hspace*{9cm}$\overbrace{\hspace{3.5em}}^{\displaystyle m}$
			
			\centering
			
			\ytableausetup{centertableaux,boxsize=1.2em}
			\begin{tabular}{r@{}l}
				\raisebox{-5.5ex}{$n \left\{\vphantom{\begin{array}{c}~\\[2.2ex] ~
					\end{array}}\right.$} &
				\begin{ytableau}
					~               &       &       & \none[\dots]  &  & \star & \none[\dots] & \star  \\ 
					~               & \none & \none & \none         & \none \\
					\none[\vdots]   & \none & \none & \none         & \none \\
					~               & \none & \none & \none         & \none \\
					~                \star  \\ \none[\vdots] \\ \star
				\end{ytableau}
				\\[-0.5ex]				
			\end{tabular}
			\caption{}
			\label{yd}
		\end{figure}
		
		Each box in Figure \ref{yd} represents a node and the size of a block indicates the total number of nodes being added (removed). 
	\end{Def}	
	\begin{ex}
		Let us take a partition  $\lambda(2,1) =$ 
		\ytableausetup{centertableaux,boxsize =0.7em}
		\begin{ytableau}
			\none & & \\ \none &
		\end{ytableau} in $V^{1}_{1}$, add a block $B_{2-t,t},~0\leq t\leq 2$ to this partition. 
		
		We get 
		$\lambda(5,1)$ = \ytableausetup{centertableaux,boxsize =0.7em}
		\begin{ytableau}
			\none & & & \star & \star \\ \none &
		\end{ytableau},
		$\lambda(5,2)$ = \ytableausetup{centertableaux}
		\begin{ytableau}
			\none & & & \star \\ \none & \\ \none & \star
		\end{ytableau}, 
		$\lambda(5,3)$ =
		\begin{ytableau}
			\none  & & \\ \none & \\ \none & \star \\ \none  & \star
		\end{ytableau} in $W^{2}_{1}.$	
	\end{ex}
	\begin{note}
		In this study, we consider only the blocks of size $p^{k}(p-1),~p^{k}(p-1)^{2},~p^{k}(p-2)$, where $k\geq 0$. In our previous paper, we constructed the \textbf{\textit{p}}-Bratteli diagram using only these blocks. This study is related to our previous work and builds on it.
	\end{note}
	
	\begin{thm} \label{thm 4.1}\cite{[TH]}
		The idempotents $\{T^{(r,s)}_{l}E^{(r,s)}_{\rho^{i}}, \ 1\leq i\leq p^{s}, \ \text{and} \ (i,p^{s}) =1 \}_{0\leq l \leq p^{r-s}-1}$,  $1\leq s\leq r$, form a set of mutually orthogonal primitive idempotents that are conjugate, giving rise to an irreducible representation of degree $p^{s-1}(p-1)$, for every $l$ and for every $s$.
	\end{thm}
	\begin{thm}\label{thm 4.3}\cite{[TH]}
		The idempotents $\{T^{(r,s)}_{l}\mathbb{E}^{(r,s)}_{\rho^{i}}, \ 1\leq i\leq p^{s}, \ \text{and}\ (i,p^{s})=1\}_{0\leq l \leq p^{r-s}-1}$, $ 1\leq s\leq r-1$ form a set of mutually orthogonal primitive idempotents that are conjugate, giving rise to an irreducible representation of degree $p^{s-1}(p-1)$, for every $l$ and for every $s$.
	\end{thm}
	
	\begin{thm} \label{thm 5.1}\cite{[TH]}
		Matrix units for the irreducible representation $\lambda(p^{r-s}[2sp-(2s+1)],k)$ defined in Theorem \ref{thm 4.1} of group $G_{r}$ are given by 
		
		$(E^{(s_{r},l)}_{ij}) = \tau^{-(i-1)}E^{(s_{r},l)}_{1}\tau^{(i-1)}\tau^{j-i}, \ 1\leq i,j\leq p^{s-1}(p-1)$ for $1\leq s \leq r$,
		
		\noindent where $E^{(s_{r},l)}_{1} = T^{(r,s)}_{l}E^{(r,s)}_{\rho}$ for $1\leq s \leq r$ and $0\leq l< p^{r-s}$.
	\end{thm}
	\begin{thm}\label{thm 5.3}\cite{[TH]}
		Matrix units for the irreducible representation $\lambda(p^{r-s-1}[(2s+1)p-2(s+1)],k)$ defined in Theorem \ref{thm 4.3} of group $SG_{r}$ are given by 
		
		$(\mathbb{E}^{(s_{r},l)}_{ij}) = \tau^{-(i-1)}\mathbb{E}^{(s_{r},l)}_{1}\tau^{(i-1)}\tau^{(j-i)}, \ 1\leq i,j\leq p^{s-1}(p-1)$ for $1\leq s \leq r-1$,
		
		\noindent where $\mathbb{E}^{(s_{r},l)}_{1} = T^{(r,s)}_{l}\mathbb{E}^{(r,s)}_{\rho}$ for $1\leq s \leq r-1$ and $0\leq l < p^{r-s}$.
	\end{thm}
	The following remark is an immediate consequence of Theorem 3.4.1 in \cite{[TH]}.
	\begin{rem}\label{rem1}
		\begin{enumerate}
			\item[]
			\item Let $\lambda(p^{r-s}(2sp-(2s+1)),k^{'})\in V^{r}_{s}$, and let $k=p^{r-s}(2sp-(2s+1))$. Then
			
			$\lambda(k,\frac{k-1}{2}\pm p^{r-s-1}j_{1}\pm p^{r-s-2}j_{2}\pm \dots \pm pj_{r-s-1}\pm j_{r-s})$ represents $T^{(r,s)}_{l}E^{(r,s)}_{\rho^{j}},$ 
			
			\noindent where 
			
			$k^{'}=\frac{k-1}{2}\pm p^{r-s-1}j_{1}\pm p^{r-s-2}j_{2}\pm \dots \pm pj_{r-s-1}\pm j_{r-s}$ and
			$l=(p-1)\displaystyle\sum_{i=1}^{r-s}p^{r-s-i}\nu(j_{i}),$ with
			$\nu(j_{i})
			=\begin{cases}
				2j_{i}, ~\text{if the coefficient of} ~p^{r-s-i}j_{i} ~\text{is}~ -1 \\
				2j_{i}-1,~ \text{if the coefficient of}~ p^{r-s-i}j_{i} ~\text{is}~ 1\\
				0, ~\text{if} ~j_{i}=0
			\end{cases}$ and
			$0\leq j_{i}\leq \frac{p-1}{2}$ for $i=1,\dots,r-s$. \label{G}
			\item Let $\lambda(p^{r-s}((2s+1)p-(2s+2)),k^{'})\in W^{r+1}_{s}$ and $k=p^{r-s}((2s+1)p-(2s+2))$. Then
			
			$\lambda(k,\frac{k-1}{2}\pm p^{r-s}j_{1}\pm p^{r-s-1}j_{2}\pm \dots \pm pj_{r-s}\pm j_{r-s+1})$ represents $T^{(r+1,s)}_{l}\mathbb{E}^{(r+1,s)}_{\rho^{j}},$ 
			
			\noindent where 
			
			$k^{'}=\frac{k-1}{2}\pm p^{r-s}j_{1}\pm p^{r-s-1}j_{2}\pm \dots \pm pj_{r-s}\pm j_{r-s+1}$ and $l=(p-1)\displaystyle\sum_{i=1}^{r-s+1}p^{r-s+1-i}\nu(j_{i}),$ with 
			$\nu(j_{i})
			=\begin{cases}
				2j_{i}, ~\text{if the coefficient of} ~p^{r-s-i}j_{i} ~\text{is}~ -1 \\
				2j_{i}-1,~ \text{if the coefficient of}~ p^{r-s-i}j_{i} ~\text{is}~ 1\\
				0, ~\text{if} ~j_{i}=0
			\end{cases}$ and
			$0\leq j_{i}\leq \frac{p-1}{2}$ for $i=1,\dots,r-s+1$. \label{SG}	
			\item Let $\lambda(p^{r-1}(p-1),k^{'})$, with $~0\leq k^{'}<p^{r-1}(p-1)$, be a hook partitions in $V^{r}_{0}$ or $W^{r}_{0}$. Then the hook partition
			$$\lambda(p^{r-1}(p-1),k^{'})~\text{ represents}~T^{(r)}_{j}E^{(r,0)}_{\rho^{i}}.$$ \label{Common}
		\end{enumerate}
	\end{rem}
	\subsection*{Schur function:} Let $\lambda/ \mu$ be a skew shape. The skew Schur function $s_{\lambda/\mu}(x)$ of shape $\lambda/\mu$ in variables $x=(x_{1},x_{2},x_{3},\dots)$ is the formal power series $s_{\lambda/\mu}(x)=\sum_{T}x^{T}$, summed over all SSYTs T of shape $\lambda/\mu$. If $\mu = \emptyset$ so $\lambda/\mu =\lambda$, then we call $s_{\lambda}(x)$ the Schur function of shape $\lambda$ \cite{[RS]}.
	\subsection*{Cauchy identity:} 
	$$\prod_{i,j}\frac{1}{1-x_{i}y_{j}} = \sum_{\lambda}s_{\lambda(x)}s_{\lambda(y)}$$
	where $s_{\lambda}(x),~s_{\lambda}(y)$ are Schur functions of shape $\lambda$ \cite{[RS]}.
	
		\section{Bijection between the elements of $G_{r}$ (and $SG_{r}$) and the pair of standard \textbf{\textit{p}}-Young tableaux}
	Using matrix units, we develop the Robinson-Schensted correspondence between the elements of the groups $G_{r}$ and $SG_{r}$ and the pair of standard \textbf{\textit{p}}-Young tableaux using matrix units. To construct this correspondence, we assign integers to blocks of size $p^{k}(p-1),~k\geq 0$. This aids in numbering each standard \textbf{\textit{p}}-Young tableau and plays a key role in defining the bijection. We discuss a few properties of the Robinson-Schensted correspondence in our case. First, we examine the element corresponding to a pair of standard \textit{\textbf{p}}-Young tableaux $(P,Q)$ and $(Q,P)$ \cite{[BS]}. Then, we consider the Knuth correspondence and the dual Knuth correspondence \cite{[BS]}.
	\begin{Def}\label{std}
		\begin{itemize}
			\item[]
			\item The \textbf{\textit{p}}-Young diagram of the hook partitions of shape $\lambda(p^{r-1}(p-1),k)\in V^{r}_{0}(\text{or}~W^{r}_{0})$, where $0\leq k <p^{r-1}(p-1)$, consists of blocks of size $p^{n}(p-1)^{2}$, for $0\leq n\leq r-2$, following the induction and restriction defined in \cite{[TH]}  and by Definition \ref{block}.  
			
			\item The standard \textbf{\textit{p}}-Young tableau corresponding to the one-dimensional representation of shape $\lambda(p^{r-1}(p-1),k)\in V^{r}_{0}$ (or $ W^{r}_{0}$) is obtained by filling every block with the same integer.
			\item The \textbf{\textit{p}}-Young diagram of the hook partitions of shape $\lambda(p^{r-s}[2sp-(2s+1)],i)\in V^{r}_{s}$, for $1\leq s\leq r$, $\left(\text{or}~\lambda(p^{r-s-1}[(2s+1)p-2(s+1)],j)\in W^{r}_{s},~ \text{for}~0\leq s\leq r-1\right)$ consists of blocks of size $p^{r-s}(p-1)\left(\text{or}~p^{r-s-1}(p-1)\right)$ and $p^{r-s}(p-2)\left(\text{or}~p^{r-s-1}(p-2)\right)$ respectively, following the induction and restriction defined in \cite{[TH]} and by Definition \ref{block}, respectively.  
			
			In other words, every \textbf{\textit{p}}-Young tableau corresponds to the path in the Bratteli diagram that starts at $\lambda(p^{r-s}[2sp-(2s+1)],i)\in V^{r}_{s}$, for $1\leq s \leq r$, $(\text{or}~\lambda(p^{r-s-1}[(2s+1)p-2(s+1)],j)\in W_{s}^{r}$, for $1\leq s\leq r-1)$ and ends at $\lambda(p-1,k)$ for some $0\leq k<p-1$.
			
			\item The standard \textbf{\textit{p}}-Young tableau \textbf{$(\textbf{\textit{SYT}}_{\textbf{\textit{p}}})$} is a \textbf{\textit{p}}-Young diagram obtained by filling each block with a single positive integer, such that the integers strictly increase along the blocks in the first row and along the blocks in the first column. 
		\end{itemize}
	\end{Def}
		\begin{ex}
		The	$\textbf{\textit{SYT}}_{\textbf{3}}$ of shape $\lambda(6,k),$ for $0\leq k <6$ such that $\lambda(6,k)\in V^{2}_{0},$ in the group $G_{2}$  is given as follows:
		
		\ytableausetup{centertableaux,boxsize=1em}	
		$\lambda(6,0)$ = \ytableausetup{centertableaux}
		\begin{ytableau}
			1 & 1 & 1 & 1 & 1 & 1
		\end{ytableau},
		$\lambda(6,1)$ = \ytableausetup{centertableaux}
		\begin{ytableau}
			1 & 1 & 1 & 1 & 1 \\ 1
		\end{ytableau},
		$\lambda(6,2)$ = \ytableausetup{centertableaux}
		\begin{ytableau}
			1 & 1 & 1 & 1 \\ 1 \\ 1
		\end{ytableau},
	
		$\lambda(6,3)$ = \ytableausetup{centertableaux}
		\begin{ytableau}
			1 & 1 & 1 \\ 1 \\ 1 \\ 1
		\end{ytableau},
			$\lambda(6,4)$ = \ytableausetup{centertableaux}
		\begin{ytableau}
			1 & 1 \\ 1 \\ 1 \\ 1 \\ 1
		\end{ytableau},
		$\lambda(6,5)$ = \ytableausetup{centertableaux}
		\begin{ytableau}
			1 \\ 1 \\ 1 \\ 1 \\ 1 \\ 1
		\end{ytableau}
		
	\end{ex}
	\begin{ex}
		Let  $\lambda(2,1) =$ 
		\ytableausetup{centertableaux,boxsize=1.2em}
		\begin{ytableau}
			1 & 1 \\  2
		\end{ytableau} $\in V^{1}_{1},$ be a $\textbf{\textit{SYT}}_{\textbf{3}}$ in the group $G_{1}.$ Now, by adding a block $B_{2-t,t}$ with $0\leq t\leq 2$ to this $\textbf{\textit{SYT}}_{\textbf{3}}$, we obtain
		
		$\lambda(5,1)$ = \ytableausetup{centertableaux}
		\begin{ytableau}
			1 & 1 & 3 & 3 \\ 2
		\end{ytableau},
		$\lambda(5,2)$ = \ytableausetup{centertableaux}
		\begin{ytableau}
			1& 1 & 3 \\  2 \\ 3
		\end{ytableau}, 
		$\lambda(5,3)$ =
		\begin{ytableau}
			1 & 1  \\ 2  \\ 3 \\ 3
		\end{ytableau} in $W^{2}_{1}.$	
	\end{ex}			
	
	\begin{Def}\label{Definition pob}
		We assign numbers to the block types of size $p^{k}(p-1)$ as follows:
		
		The block of type $B_{\frac{p^{k}(p-1)}{2}+p^{k}m,\frac{p^{k}(p-1)}{2}-p^{k}m}$ is numbered as $2m+1$, where $0\leq m\leq \frac{p-1}{2}$.
		
		The block of type $B_{\frac{p^{k}(p-1)}{2}-p^{k}m,\frac{p^{k}(p-1)}{2}+p^{k}m}$ is numbered as $2m$, where $1\leq m\leq \frac{p-1}{2}$.

	\end{Def}
	
	\subsection{Robinson-Schensted correspondence}
	We develop the Robinson-Schensted correspondence between the pairs of $\textbf{\textit{SYT}}_{\textbf{\textit{p}}}$ and the elements of the groups $G_{r}$ and $SG_{r}$.
	\subsubsection*{Assigning integer to standard $p$-Young tableau}
	To assign an integer to every $\textbf{\textit{SYT}}_{\textbf{\textit{p}}}$, we utilize the numbering defined in Definition \ref{Definition pob}. This process involves considering two distinct cases to assign an integer value appropriately.
	
	\begin{description}
		\item[Case (a)] Consider a $\textbf{\textit{SYT}}_{\textbf{\textit{p}}}$ of shape $\lambda(2rp-(2r+1),r(p-1)-1)\in V_{r}^{r}$, corresponding to an irreducible representation of degree $p^{r-1}(p-1)$ of the group $G_{r}$.
		
		\begin{description}
			\item[Step 1] Look at the block of type $B_{p-1-t,t}$, where $0\leq t\leq p-1$, which contains the last entry (i.e., highest integer occurring in the $\textbf{\textit{SYT}}_{\textbf{\textit{p}}}$) of the $\textbf{\textit{SYT}}_{\textbf{\textit{p}}}$ of shape $\lambda(2rp-(2r+1),r(p-1)-1)\in V_{r}^{r}$. By Definition \ref{Definition pob}, this block type is numbered as $'i_{1}'$. Subsequently, the block is removed and the $\textbf{\textit{SYT}}_{\textbf{\textit{p}}}$ is reduced to a $\textbf{\textit{SYT}}_{\textbf{\textit{p}}}$ of shape $\lambda((2r-1)p-2r,k^{'})\in W_{r-1}^{r}$. After removing the last entry of the $\textbf{\textit{SYT}}_{\textbf{\textit{p}}}$ of the shape $\lambda((2r-1)p-2r,k^{'})\in W_{r-1}^{r}$, the tableau is reduced further to a $\textbf{\textit{SYT}}_{\textbf{\textit{p}}}$ of shape $\lambda((2r-2)p-(2r-1),(r-1)p-r)\in V_{r-1}^{r-1}$.
			
			\begin{figure}[!ht]
				\vspace{-0.4cm}
				\begin{center}
					\includegraphics[width=9cm, height=5cm]{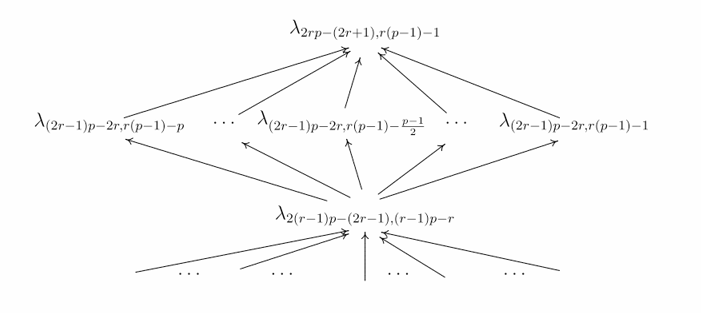}
				\end{center}
				\vspace{-0.9cm}
			\end{figure}
			
			\item[Step 2]	
			Once again consider the block of type $B_{p-1-t,t}$, where $0\leq t \leq p-1$, which contains the last entry of the reduced tableau. By Definition \ref{Definition pob}, this block type is numbered as $'i_{2}'$. This block is then removed as well.			
			This process continues by successively removing such blocks. Now consider the reduced $\textbf{\textit{SYT}}_{\textbf{\textit{p}}}$ of shape $\lambda(2(r-m+1)p-(2(r-m+1)+1),(r-m+1)(p-1)-1)\in V^{r-m+1}_{r-m+1}$, and look at the block of type $B_{p-1-t,t}$, where $0\leq t \leq p-1$, which contains the last entry of the $\textbf{\textit{SYT}}_{\textbf{\textit{p}}}$. By Definition \ref{Definition pob}, this block type is numbered as $'i_{m}'$, and it is subsequently removed. As a result, the $\textbf{\textit{SYT}}_{\textbf{\textit{p}}}$ is reduced to the $\textbf{\textit{SYT}}_{\textbf{\textit{p}}}$ of shape $\lambda((2(r-m+1)-1)p-2(r-m+1),k^{''})\in W^{r-m}_{r-m+1}$, from which the block containing the last entry is also removed. 
			
			This $\textbf{\textit{SYT}}_{\textbf{\textit{p}}}$ is reduced to $\lambda(2(r-m)p-(2(r-m)+1),(r-m)(p-1)-1)\in V^{r-m}_{r-m}$.
			Now, look at the block of type $B_{p-1-t,t},~0\leq t \leq p-1$, which contains the last entry of the $\textbf{\textit{SYT}}_{\textbf{\textit{p}}}$. By Definition \ref{Definition pob}, this block type is numbered as $'i_{m+1}'$ for $1\leq m \leq r-2$. This yields a sequence of integers  $i_{1},i_{2},\dots,i_{r-1}$, $1\leq i_{1},i_{2},\dots,i_{r-1}\leq p$.
			
			\item[Step 3] After removing all the blocks, as in the previous steps, we reach a stage where the partition has the shape $\lambda(2p-3,p-2)$ with dimension $p-1$. Consider the path from $\lambda(2p-3,p-2)$ to $\lambda(p-1,t_{1})$ by removing $t_{1}$ nodes horizontally and $p-2-t_{1}$ nodes vertically, where $0\leq t_{1}\leq p-2$. We assign numbers to these paths as follows:	
			$$\lambda(2p-3,p-2)\hookleftarrow \lambda(p-1,t_{1}) : 2t_{1}+1,~0\leq t_{1}<\frac{p-1}{2}$$
			and its conjugate is numbered as
			$$\lambda(2p-3,p-2)\hookleftarrow \lambda(p-1,p-2-t_{1}) : 2(t_{1}+1)$$
			\begin{figure}[!ht]
				\vspace{-0.5cm}
				\begin{center}
					\includegraphics[width=9cm, height=2.5cm]{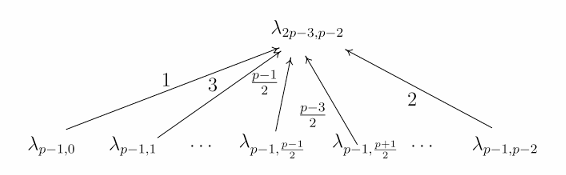}
				\end{center}
				\vspace{-0.5cm}
			\end{figure}
			
			Let the numbering be, say, $i_{r}$.
			
			\item[Step 4]	
			Let	$j=
			i_{1}+p(i_{2}-1)+p^{2}(i_{3}-1)+\dots+p^{r-2}(i_{r-1}-1)+p^{r-1}i_{r} ~ \text{for} ~ r\geq1$.
			
			Then $1\leq j\leq p^{r-1}(p-1)$. The path corresponding to the $\textbf{\textit{SYT}}_{\textbf{\textit{p}}}$ is denoted by $P_{j}$.
			
			Take a $\textbf{\textit{SYT}}_{\textbf{\textit{p}}}$ of the shape $\lambda((2r-1)p-2r,k^{'})\in W_{r-1}^{r}$, $(r-1)p-r\leq k^{'} \leq rp-(r+1)$. Let $k = (2r-1)p-2r$, and the figure below is an example by fixing $k^{'} =\frac{k-1}{2}$. For any fixed $k^{'}$, the paths are the same as those for the partition $\lambda(k-(p-1),\frac{k-(p-1)-1}{2})\in V^{r-1}_{r-1}$. First, the block that contains the last entry is removed. Now, this case is reduced to the previous case and by doing the same procedure for the reduced tableau in $V^{r-1}_{r-1}$, we obtain the following integers $i_{1},i_{2},\dots,i_{r-1}$ where $1\leq i_{1},i_{2},\dots,i_{r-2}\leq p$ and $0\leq i_{r-1}\leq p-2$.
			
			\begin{figure}[!ht]
				\vspace{-0.3cm}
				\begin{center}
					\includegraphics[width=9cm, height=3cm]{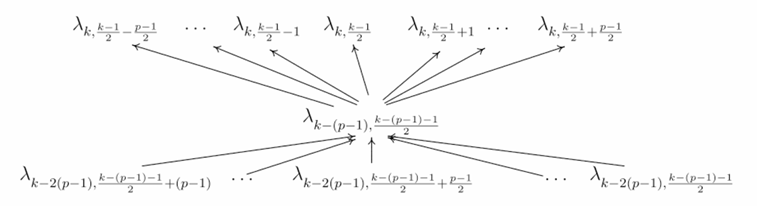}
				\end{center}
				\vspace{-0.5cm}
			\end{figure}
			
			Let $j=i_{1}+p(i_{2}-1)+p^{2}(i_{3}-1)+\dots+p^{r-3}(i_{r-2}-1)+p^{r-2}i_{r-1},~ \text{for} ~ r>1$.
			\noindent Then $1\leq j\leq p^{r-2}(p-1)$. The path corresponding to the $\textbf{\textit{SYT}}_{\textbf{\textit{p}}}$ is denoted by $P_{j}$.
		\end{description}	
		\item[Case (b)]	
		Consider a $\textbf{\textit{SYT}}_{\textbf{\textit{p}}}$ of shape $\lambda(p^{r-s}[2sp-(2s+1)],l_{1})\in V_{s}^{r}$, where $1\leq s\leq r-1$, and $p^{r-s}[sp-(s+1)]\leq l_{1}<sp^{r-s}(p-1)$, corresponding to an irreducible representation of degree $p^{s-1}(p-1)$ of the group $G_{r}$. Let us take $k=p^{r-s}[2sp-(2s+1)]$.
		\begin{description}
			\item[Step 1]	
			Let $l_{1}$ be fixed. Consider the block of type $B_{p^{r-s}(p-1-t),p^{r-s}t}$, where $0\leq t\leq p-1$, which contains the last entry of the $\textbf{\textit{SYT}}_{\textbf{\textit{p}}}$ of shape $\lambda(p^{r-s}[2sp-(2s+1)],l_{1})\in V_{s}^{r}$. By Definition \ref{Definition pob}, this block type is numbered as $'i_{1}'$. Subsequently, the block is removed. Now, this $\textbf{\textit{SYT}}_{\textbf{\textit{p}}}$ is reduced to a  $\textbf{\textit{SYT}}_{\textbf{\textit{p}}}$ of shape $\lambda(p^{r-s}[(2s-1)p-2s],l^{'}_{1})\in W_{s-1}^{r}$. After removing the block containing the last entry of the $\textbf{\textit{SYT}}_{\textbf{\textit{p}}}$ of shape $\lambda(p^{r-s}[(2s-1)p-2s],l^{'}_{1})\in W_{s-1}^{r}$, it is further reduced to a $\textbf{\textit{SYT}}_{\textbf{\textit{p}}}$ of shape $\lambda(p^{r-s}[2(s-1)p-(2s-1)],l_{2})$ in $V_{s-1}^{r-1}$.	
			\begin{figure}[!ht]
				\begin{center}
					\includegraphics[width=9cm, height=4cm]{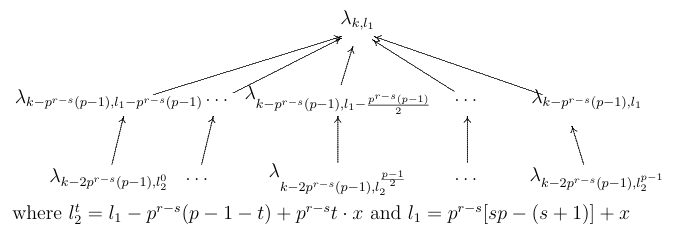}
				\end{center}
				\vspace{-0.5cm}
			\end{figure}
			
			\item[Step 2]	
			Once again, consider the block of type $B_{p^{r-s}(p-1-t),p^{r-s}t}$, where $0\leq t\leq p-1$, which contains the last entry in the reduced tableau. By Definition \ref{block}, this block type is numbered as $'i_{2}'$. This block is then removed.
			
			This process can be continued by successively removing blocks. Now, consider the reduced $\textbf{\textit{SYT}}_{\textbf{\textit{p}}}$ of shape $\lambda(p^{r-s}[2(s-m+1)p-(2(s-m+1)+1)],l_{m})\in V_{s-m+1}^{r-m+1}$. Look at the block of type $B_{p^{r-s}(p-1-t),p^{r-s}t}$, with $0\leq t\leq p-1$, which contains the last entry of the $\textbf{\textit{SYT}}_{\textbf{\textit{p}}}$. By Definition \ref{block}, this block type is numbered as $'i_{m}'$. Consequently, the block is removed. Now, this $SYT_{p}$ is reduced to the $\textbf{\textit{SYT}}_{\textbf{\textit{p}}}$ of shape $\lambda(p^{r-s}[(2(s-m+1)-1)p-2(s-m+1)],l^{'}_{m})\in W_{s-m}^{r-m+1}$, and the block containing the last entry is removed. 
			
			This $\textbf{\textit{SYT}}_{\textbf{\textit{p}}}$ is further reduced to $\lambda(p^{r-s}[2(s-m)p-2(s-m)+1],l_{m+1})\in V_{s-m}^{r-m}$. Now, look at the block of type $B_{p^{r-s}(p-1-t),p^{r-s}t}$, with $0\leq t\leq p-1$, which contains the last entry of the $\textbf{\textit{SYT}}_{\textbf{\textit{p}}}$. By Definition \ref{block}, this block type is numbered as  $'i_{m+1}'$ for $0\leq m\leq s-2$. This yields a sequence of integers
			$i_{1},i_{2},\dots,i_{s-1}$ where $1\leq i_{1},i_{2},\dots,i_{s-1}\leq p$. 
			
			\item[Step 3]	
			After removing all the blocks, we arrive at a stage where the partition takes the shape of $\lambda(p^{r-s}(2p-3),k^{'})$ with dimension $p-1$. Consider the path from $\lambda(p^{r-s}(2p-3),k^{'})$ to $\lambda(p^{r-s}(p-1),l^{'})$. We number these paths as follows:
			
			$\lambda(p^{r-s}(2p-3),k^{'})\hookleftarrow \lambda(p^{r-s}(p-1),k^{'}-p^{r-s}(p-2-t)) ~\text{numbered as}~j^{'}~\text{if}~j^{'}p^{r-s}\leq k^{'}-p^{r-s}(p-2-t)<(j^{'}+1)p^{r-s}. $
			
			\item[Step 4]	
			Let	
			
			$j=\begin{cases}
				1+j^{'}+p^{s-2}(p-1)(i_{1}-1)+\sum_{n=1}^{s-2}p^{n-1}(p-1)(i_{n+1}-1) & \text{for} ~ s>1\\
				j^{'}+1 & \text{for} ~ s=1
			\end{cases}$
			
			Then $1\leq j\leq p^{s-1}(p-1)$. The path corresponding to the $\textbf{\textit{SYT}}_{\textbf{\textit{p}}}$ is denoted by $P_{j}$.

			Consider a $\textbf{\textit{SYT}}_{\textbf{\textit{p}}}$ of shape $\lambda(p^{r-s}[(2s-1)p-2s],k^{'})\in W_{s-1}^{r}$, with $p^{r-s}[(s-1)p-s]\leq k^{'} \leq sp^{r-s}(p-1)$. Let $k = p^{r-s}[(2s-1)p-2s]$.	
			For any fixed $k^{'}$, the paths are the same as those in the partition 
			$\lambda(k-p^{r-s}(p-1),\frac{k-p^{r-s}(p-1)-1}{2})\in V^{r-1}_{s-1}$. Begin by removing the block that contains the last entry of the $\textbf{\textit{SYT}}_{\textbf{\textit{p}}}$. This case then reduces to the previous one, so we obtain the integers $i_{1},i_{2},\dots,i_{s-2}$ where $1\leq i_{1},i_{2},\dots,i_{s-2}\leq p$. Let
			
			$j=\begin{cases}
				1+j^{'}+p^{s-3}(p-1)(i_{1}-1)+\sum_{n=1}^{s-3}p^{n-1}(p-1)(i_{n+1}-1) & \text{for} ~ s>1\\
				j^{'}+1 & \text{for} ~ s=1
			\end{cases}$
			
			Then $1\leq j\leq p^{s-2}(p-1)$. The path corresponding to the $\textbf{\textit{SYT}}_{\textbf{\textit{p}}}$ is denoted by $P_{j}$.
		\end{description}	
	\end{description}
	
	\begin{Not}
		\item $\Lambda_{1} = \{V_{0}^{r}\cup V^{r}_{s}/r\geq 1,~1\leq s\leq r\}$, $\Lambda_{2} = \{W_{0}^{r}\cup W^{r}_{s}/r\geq 1,~1\leq s\leq r-1\}$
		\item $\Lambda = \Lambda_{1}\cup \Lambda_{2}$, $\tau^{*}=\tau^{\frac{p^{r-1}(p-1)}{2}}$ 
	\end{Not}
	\subsection*{Defining the mapping between the group $G_{r}$ (and $SG_{r}$) and pairs of $\textbf{\textit{SYT}}_{\textbf{\textit{p}}}$}
	Now, construct the map $h: \bigcup\limits_{\lambda \in \Lambda_{1}}\textbf{\textit{SYT}}_{\textbf{\textit{p}}}(\lambda)\times \textbf{\textit{SYT}}_{\textbf{\textit{p}}}(\lambda)\rightarrow G_{r}$. Consider an element $g\in G_{r}$ of order $p^{r}$. 
	\begin{description}
		\item[Case (a)] 	
		Let $P_{i},P_{j}\in \mathbb{T}^{p}_{\lambda}$, where $\mathbb{T}^{p}_{\lambda}$ denotes the set of all $\textbf{\textit{SYT}}_{\textbf{\textit{p}}}$ of shape $\lambda(k,\frac{k-1}{2})$, with $k=2rp-(2r+1)$. Based on the matrix units of Theorem \ref{thm 5.1}, we define $$h\bigl((P_{i},P_{j})\bigr)=\tau^{-(i-1)}g\tau^{*}\tau^{j-1},$$ where $1\leq i,j \leq p^{r-1}(p-1)$.  
		
		\item[Case (b)] For any $1\leq s\leq r-1$, we have:
		\begin{description}	
			\item[Type 1] 
			Let $P^{0}_{i},P^{0}_{j}\in \mathbb{T}^{P}_{\lambda(0)}$ where $\mathbb{T}^{p}_{\lambda(0)}$ represents the set of all $\textbf{\textit{SYT}}_{\textbf{\textit{p}}}$ of the shape $\lambda(k,\frac{k-1}{2})$, with $k=p^{r-s}[2sp-(2s+1)]$. Based on the matrix units of Theorem \ref{thm 5.1}, we define $$h\bigl((P^{0}_{i},P^{0}_{j})\bigr)=
			\tau^{-(i-1)}g^{p^{r-s}}\tau^{*}\tau^{j-1}, $$ where $1\leq i,j \leq p^{s-1}(p-1)$.
			
			\item[Type 2] Let $P^{m^{-}}_{i},P^{m^{-}}_{j}\in \mathbb{T}^{p}_{\lambda(m^{-})}$, where $\mathbb{T}^{p}_{\lambda(m^{-})}$ denotes the set of all $\textbf{\textit{SYT}}_{\textbf{\textit{p}}}$ of shape $\lambda(k,\frac{k-1}{2}-m)$, with $k=p^{r-s}[2sp-(2s+1)]$ and $1\leq m \leq \frac{p^{r-s}-1}{2}$. 
			Based on the matrix units of Theorem \ref{thm 5.1}, we define
			$$h\bigl((P^{m^{-}}_{i},P^{m^{-}}_{j})\bigr)=\tau^{-(i-1)}g^{p^{r-s}}\tau^{*}\tau^{j-1}\tau^{lp^{s-1}(p-1)},$$
			where $1\leq i,j \leq p^{s-1}(p-1)$, and $T^{(r,s)}_{l}E^{(r,s)}_{\rho^{j^{'}}}$ is the primitive idempotent indexed by the hook partition $\lambda(k,\frac{k-1}{2}-m)$ as in Remark \ref{rem1} (\ref{G}).
			
			\item[Type 3] Let $P^{m^{+}}_{i},P^{m^{+}}_{j}\in \mathbb{T}^{p}_{\lambda(m^{+})}$, where $\mathbb{T}^{p}_{\lambda(m^{+})}$ represents the set of all $\textbf{\textit{SYT}}_{\textbf{\textit{p}}}$ of the shape $\lambda(k,\frac{k-1}{2}+m)$, with $k=p^{r-s}[2sp-(2s+1)]$ and $1\leq m \leq \frac{p^{r-s}-1}{2}$.
			Based on the matrix units of Theorem \ref{thm 5.1}, we get $$h\bigl((P^{m^{+}}_{i},P^{m^{+}}_{j})\bigr)=	\tau^{-(i-1)}g^{p^{r-s}}\tau^{*}\tau^{j-1}\tau^{lp^{s-1}(p-1)}, $$
			where $1\leq i,j \leq p^{s-1}(p-1)$, and $T^{(r,s)}_{l}E^{(r,s)}_{\rho^{j^{'}}}$ is the primitive idempotent indexed by the hook partition $\lambda(k,\frac{k-1}{2}+m)$ as in Remark \ref{rem1} (\ref{G}).
		\end{description}
	\end{description}	
	\begin{thm}
		The map $h: \bigcup\limits_{\lambda \in \Lambda_{1}}\textbf{\textit{SYT}}_{\textbf{\textit{p}}}(\lambda)\times \textbf{\textit{SYT}}_{\textbf{\textit{p}}}(\lambda)\rightarrow G_{r}$ establishes a bijection between the pairs of standard $p$-Young tableaux of the same shape and elements in the group $G_{r}$.
	\end{thm}
	\begin{proof}
		Let $g^{m}\tau^{k}\in G_{r}$. There exist $i$ and $k^{'}$ such that $m=p^{r-s}t^{i}$ and $k=lp^{s-1}(p-1)+k^{'}$ where $1\leq i\leq p^{s-1}(p-1)$, $1\leq k^{'}\leq p^{s-1}(p-1)$ and $1\leq s\leq r$. Now we have 
		\begin{align*}
			g^{m}\tau^{k}&= g^{p^{r-s}t^{i}}\tau^{lp^{s-1}(p-1)+k^{'}}=\tau^{-i}g^{p^{r-s}}\tau^{*}\tau^{i}\tau^{lp^{s-1}(p-1)+k^{'}}\tau^{\frac{p^{r-1}(p-1)}{2}}\\ &=\tau^{-i}g^{p^{r-s}}\tau^{*}\tau^{i+\frac{p^{r-1}(p-1)}{2}+k^{'}}\tau^{lp^{s-1}(p-1)}=\tau^{-((i+1)-1)}g^{p^{r-s}}\tau^{*}\tau^{(j+1)-1}\tau^{lp^{s-1}(p-1)}
		\end{align*}
		where $j=i+\frac{p^{r-1}(p-1)}{2}+k^{'}$. Let $P_{i+1}$ be a $\textbf{\textit{SYT}}_{\textbf{\textit{p}}}$ indexed by the integer numbering $i+1$, and $P_{j+1}$ be a $\textbf{\textit{SYT}}_{\textbf{\textit{p}}}$ indexed by $j+1$. Hence, the pair of $\textbf{\textit{SYT}}_{\textbf{\textit{p}}}$ $(P_{i+1},P_{j+1})$ is the inverse image of $g^{m}\tau^{k}$. 
		
		Consider the pair of tableaux $(P_{j},P_{j})$ of the same shape $\lambda(p^{r-1}(p-1),k^{'}) \in V^{r}_{0}$. This pair corresponds to the unique element $\tau^{j}$, where the relation between $k^{'}$ and $j$ is described in Remark \ref{rem1} (\ref{Common}). Hence, the bijection follows.
	\end{proof}
		\begin{ex}
		Now we discuss the integer numbering for the $\textbf{\textit{SYT}}_{\textbf{3}}$ of shape $\lambda(11,5)\in V^{3}_{3}$ in the group $G_{3}.$ 
		\begin{align*}
			\ytableausetup{boxsize =1em}
			P_{1} &=\begin{ytableau}
				1 & 1 & 3 & 4 & 5 & 6   \\ 2 \\ 3  \\ 4  \\ 5 \\ 6 
			\end{ytableau} & 	P_{2} &=\begin{ytableau}
				1 & 1 & 3 & 4 & 5 & 5  \\ 2 \\ 3  \\ 4  \\ 6 \\ 6 
			\end{ytableau} &
			P_{3} &=\begin{ytableau}
				1 & 1 & 3 & 4 & 6 & 6  \\ 2 \\ 3  \\ 4  \\ 5 \\ 5
			\end{ytableau} 
					\end{align*}
		$i_{1}=1,~i_{2}=1,~i_{3}=1\hspace*{2cm}i_{1}=2,~i_{2}=1,~i_{3}=1\hspace*{2cm}i_{1}=3,~i_{2}=1,~i_{3}=1$
		
		\noindent$j=1+3\cdot 0+3^{2}\cdot 0 =1\hspace*{1.2cm}j=2+3\cdot 0+3^{2}\cdot 0 =2\hspace*{1.2cm}j=3+3\cdot 0+3^{2}\cdot 0 =3$
		
		$(P_{1},P_{1})\mapsto g\tau^{9}\hspace*{3cm}(P_{2},P_{2})\mapsto g^{2}\tau^{9}\hspace*{3cm}(P_{3},P_{3})\mapsto g^{4}\tau^{9}$
		
		\begin{align*}
			\ytableausetup{boxsize =1em}
			P_{4} &=\begin{ytableau}
				1 & 1 & 3 & 3 & 5 & 6   \\ 2 \\ 4  \\ 4  \\ 5 \\ 6 
			\end{ytableau} &
			P_{5} &=\begin{ytableau}
				1 & 1 & 3 & 3 & 5 & 5 \\ 2 \\ 4  \\ 4  \\ 6 \\ 6 
			\end{ytableau} & 
			P_{6} &=\begin{ytableau}
				1 & 1 & 3 & 3 & 6 & 6  \\ 2 \\ 4 \\ 4  \\ 5 \\ 5 
			\end{ytableau} 
		\end{align*}
		$i_{1}=1,~i_{2}=2,~i_{3}=1\hspace*{2cm}i_{1}=2,~i_{2}=2,~i_{3}=1\hspace*{2cm}i_{1}=3,~i_{2}=2,~i_{3}=1$
		
		$(P_{4},P_{4})\mapsto g^{8}\tau^{9}\hspace*{3cm}(P_{5},P_{5})\mapsto g^{16}\tau^{9}\hspace*{3cm}(P_{6},P_{6})\mapsto g^{5}\tau^{9}$
		\begin{align*}
			\ytableausetup{boxsize =1em}
			P_{7} &=\begin{ytableau}
				1 & 1 & 4 & 4 & 5 & 6   \\ 2 \\ 3  \\ 3  \\ 5 \\ 6
			\end{ytableau} & 
			P_{8} &=\begin{ytableau}
				1 & 1 & 4 & 4 & 5 & 5  \\ 2 \\ 3  \\ 3  \\ 6 \\ 6 
			\end{ytableau} &
			P_{9} &=\begin{ytableau}
				1 & 1 & 4 & 4 & 6 & 6  \\ 2 \\ 3  \\ 3 \\ 5 \\ 5 
			\end{ytableau} 
		\end{align*}
		$i_{1}=1,~i_{2}=3,~i_{3}=1\hspace*{2cm}i_{1}=2,~i_{2}=3,~i_{3}=1\hspace*{2cm}i_{1}=3,~i_{2}=3,~i_{3}=1$
		
		$(P_{7},P_{7})\mapsto g^{10}\tau^{9}\hspace*{3cm}(P_{8},P_{8})\mapsto g^{20}\tau^{9}\hspace*{3cm}(P_{9},P_{9})\mapsto g^{13}\tau^{9}$
		
		\begin{align*}	
			\ytableausetup{boxsize =1em}
			P_{10} &=\begin{ytableau}
				1 & 2 & 3 & 4 & 5 & 6   \\ 1 \\ 3  \\ 4  \\ 5 \\ 6 
			\end{ytableau} & 
			P_{11} &=\begin{ytableau}
				1 & 2 & 3 & 4 & 5 & 5  \\ 1 \\ 3  \\ 4  \\ 6 \\ 6 
			\end{ytableau} &
			P_{12} &=\begin{ytableau}
				1 & 2 & 3 & 4 & 6 & 6  \\ 1 \\ 3  \\ 4  \\ 5 \\ 5
			\end{ytableau} 
		\end{align*}
		$i_{1}=1,~i_{2}=1,~i_{3}=2\hspace*{2cm}i_{1}=2,~i_{2}=1,~i_{3}=2\hspace*{2cm}i_{1}=3,~i_{2}=1,~i_{3}=2$
		
		$(P_{10},P_{10})\mapsto g^{26}\tau^{9}\hspace*{2.5cm}(P_{11},P_{11})\mapsto g^{25}\tau^{9}\hspace*{2.5cm}(P_{12},P_{12})\mapsto g^{23}\tau^{9}$
		
		\begin{align*}
			\ytableausetup{boxsize =1em}
			P_{13} &=\begin{ytableau}
				1 & 2 & 3 & 3 & 5 & 6   \\ 1 \\ 4  \\ 4  \\ 5 \\ 6 
			\end{ytableau} &
			P_{14} &=\begin{ytableau}
				1 & 2 & 3 & 3 & 5 & 5 \\ 1 \\ 4  \\ 4  \\ 6 \\ 6 
			\end{ytableau} & 
			P_{15} &=\begin{ytableau}
				1 & 2 & 3 & 3 & 6 & 6  \\ 1 \\ 4 \\ 4  \\ 5 \\ 5 
			\end{ytableau} 
		\end{align*}
		$i_{1}=1,~i_{2}=2,~i_{3}=2\hspace*{2cm}i_{1}=2,~i_{2}=2,~i_{3}=2\hspace*{2cm}i_{1}=3,~i_{2}=2,~i_{3}=2$
		
		$(P_{13},P_{13})\mapsto g^{19}\tau^{9}\hspace*{2.5cm}(P_{14},P_{14})\mapsto g^{11}\tau^{9}\hspace*{2.5cm}(P_{15},P_{15})\mapsto g^{22}\tau^{9}$
		
		\begin{align*}
			\ytableausetup{boxsize =1em}
			P_{16} &=\begin{ytableau}
				1 & 2 & 4 & 4 & 5 & 6   \\ 1 \\ 3  \\ 3  \\ 5 \\ 6
			\end{ytableau} & 
			P_{17} &=\begin{ytableau}
				1 & 2 & 4 & 4 & 5 & 5  \\ 1 \\ 3  \\ 3  \\ 6 \\ 6 
			\end{ytableau} &
			P_{18} &=\begin{ytableau}
				1 & 2 & 4 & 4 & 6 & 6  \\ 1 \\ 3  \\ 3 \\ 5 \\ 5 
			\end{ytableau} 
		\end{align*}
		$i_{1}=1,~i_{2}=3,~i_{3}=3\hspace*{2cm}i_{1}=2,~i_{2}=3,~i_{3}=3\hspace*{2cm}i_{1}=3,~i_{2}=3,~i_{3}=3$
		
		$(P_{16},P_{16})\mapsto g^{17}\tau^{9}\hspace*{2.5cm}(P_{17},P_{17})\mapsto g^{7}\tau^{9}\hspace*{2.5cm}(P_{18},P_{18})\mapsto g^{14}\tau^{9}$
	\end{ex}
	
	Now we construct the map $h^{'}: \bigcup\limits_{\lambda \in \Lambda_{2}}\textbf{\textit{SYT}}_{\textbf{\textit{p}}}(\lambda)\times \textbf{\textit{SYT}}_{\textbf{\textit{p}}}(\lambda)\rightarrow SG_{r}$. Consider an element $g^{p}\in SG_{r}$ of order $p^{r-1}$.
	
	For any $1\leq s \leq r-2$, we have
	
	\noindent\textbf{Type 1:} Let $P^{0}_{i},P^{0}_{j}\in \mathbb{T}^{p}_{\lambda(0)}$, where $\mathbb{T}^{p}_{\lambda(0)}$ denotes the set of all $\textbf{\textit{SYT}}_{\textbf{\textit{p}}}$ of shape $\lambda(k,\frac{k-1}{2})$, with $k=(2s+1)p-2(s+1)$. Based on the matrix units of Theorem \ref{thm 5.3}, we define $$h^{'}\bigl((P^{0}_{i},P^{0}_{j})\bigr)=
	\tau^{-(i-1)}(g^{p})^{p^{r-s-1}}\tau^{*}\tau^{j-1},$$ where $~1\leq i,j \leq p^{s-1}(p-1)$.
	
	\noindent\textbf{Type 2:} 
	Let $P^{m^{-}}_{i},P^{m^{-}}_{j}\in \mathbb{T}^{p}_{\lambda(m^{-})}$, where $\mathbb{T}^{p}_{\lambda(m^{-})}$ represents the set of all $\textbf{\textit{SYT}}_{\textbf{\textit{p}}}$ of the shape $\lambda(k,\frac{k-1}{2}-m)$, with $k=(2s+1)p-2(s+1)$ and $1\leq m\leq \frac{p^{r-s}-1}{2}$. Based on the matrix units of Theorem \ref{thm 5.3}, we get $$h^{'}\bigl((P^{m^{-}}_{i},P^{m^{-}}_{j})\bigr)=
	\tau^{-(i-1)}(g^{p})^{p^{r-s-1}}\tau^{*}\tau^{j-1}\tau^{lp^{s-1}(p-1)},$$ where $1\leq i,j \leq p^{s-1}(p-1)$, and $T^{(r,s)}_{l}\mathbb{E}^{(r,s)}_{\rho^{j^{'}}}$ is the primitive idempotent indexed by the hook partition $\lambda(k,\frac{k-1}{2}-m)$ as in Remark \ref{rem1} (\ref{SG}).
	
	\noindent\textbf{Type 3:} Let $P^{m^{+}}_{i},P^{m^{+}}_{j}\in \mathbb{T}^{p}_{\lambda(m^{+})}$, where $\mathbb{T}^{p}_{\lambda(m^{+})}$ represents the set of all $\textbf{\textit{SYT}}_{\textbf{\textit{p}}}$ of the shape $\lambda(k,\frac{k-1}{2}+m)$, with $k=(2s+1)p-2(s+1)$ and $0\leq m\leq \frac{p^{r-s}-1}{2}$. Based on the matrix units of Theorem \ref{thm 5.3}, we get $$h^{'}\bigl((T^{p}_{\lambda(m^{+})},S^{p}_{\lambda(m^{+})})\bigr)=
	\tau^{-(i-1)}(g^{p})^{p^{r-s-1}}\tau^{*}\tau^{j-1}\tau^{lp^{s-1}(p-1)}$$
	where $1\leq i,j \leq p^{s-1}(p-1)$, and $T^{(r,s)}_{l}\mathbb{E}^{(r,s)}_{\rho^{j^{'}}}$ is the primitive idempotent indexed by the hook partition $\lambda(k,\frac{k-1}{2}+m)$ as in Remark \ref{rem1} (\ref{SG}).
	\begin{thm}
		The map $h^{'}: \bigcup\limits_{\lambda \in \Lambda_{2}}\textbf{\textit{SYT}}_{\textbf{\textit{p}}}(\lambda)\times \textbf{\textit{SYT}}_{\textbf{\textit{p}}}(\lambda)\rightarrow SG_{r}$ establishes the bijection between pairs of standard $p$-Young tableaux of the same shape and element in a group $SG_{r}$.
	\end{thm}
	\begin{proof}
		Let $g^{pm}\tau^{k}\in SG_{r}$. There exist $i$ and $k^{'}$ such that $m=p^{r-s-1}t^{i}$ and $k=lp^{s-1}(p-1)+k^{'}$, where $1\leq i\leq p^{s-1}(p-1)$, $1\leq k^{'}\leq p^{s-1}(p-1)$ and $1\leq s\leq r-1$. Now we have 
		\begin{align*}
			g^{pm}\tau^{k}&= (g^{p})^{p^{r-s-1}t^{i}}\tau^{lp^{s-1}(p-1)+k^{'}}=\tau^{-i}g^{p^{r-s}}\tau^{*}\tau^{i}\tau^{lp^{s-1}(p-1)+k^{'}}\tau^{\frac{p^{r-1}(p-1)}{2}}\\ &=\tau^{-i}g^{p^{r-s}}\tau^{*}\tau^{i+\frac{p^{r-1}(p-1)}{2}+k^{'}}\tau^{lp^{s-1}(p-1)}=\tau^{-((i+1)-1)}g^{p^{r-s}}\tau^{*}\tau^{(j+1)-1}\tau^{lp^{s-1}(p-1)}	
		\end{align*}
		where $j=i+\frac{p^{r-1}(p-1)}{2}+k^{'}$. Let $P_{i+1}$ be a $\textbf{\textit{SYT}}_{\textbf{\textit{p}}}$ indexed by the integer $i+1$ and $P_{j+1}$ be a $\textbf{\textit{SYT}}_{\textbf{\textit{p}}}$ indexed by $j+1$. Hence, the pair of $\textbf{\textit{SYT}}_{\textbf{\textit{p}}}$ $(P_{i+1},P_{j+1})$ is the inverse image of $g^{pm}\tau^{k}$. 			
		
		Consider the pair of tableaux $(P_{j},P_{j})$ of the same shape $\lambda(p^{r-1}(p-1),k^{'}) \in W^{r}_{0}$. This pair corresponds to the unique element $\tau^{j}$ where the relation between $k^{'}$ and $j$ is described in Remark \ref{rem1} (\ref{Common}). Hence, the bijection follows.
	\end{proof}
		\begin{ex}
		Let us discuss the matrix representation of the irreducible representation indexed by the hook partition $\lambda(15,k)\in W^{3}_{1},$ where $3\leq k\leq 11.$ 
		Here, $k$ and $l$ are related as described in Remark \ref{rem1} (\ref{SG}).
		
		\noindent	$\begin{pmatrix}
			g^{9}\tau^{9} & g^{9}\tau^{10}\\
			g^{18}\tau^{8}&g^{18}\tau^{9} 
		\end{pmatrix},$
		$\begin{pmatrix}
			g^{9}\tau^{17} & g^{9} \\
			g^{18}\tau^{16}&g^{18}\tau^{17} 
		\end{pmatrix},$
		$\begin{pmatrix}
			g^{9}\tau & g^{9}\tau^{2} \\
			g^{18} &g^{18}\tau 
		\end{pmatrix},$
		$\begin{pmatrix}
			g^{9}\tau^{15} & g^{9}\tau^{16} \\
			g^{18}\tau^{14}&g^{18}\tau^{15} 
		\end{pmatrix},$~ 
		$\begin{pmatrix}
			g^{9}\tau^{5} & g^{9}\tau^{6} \\
			g^{18}\tau^{6}&g^{18}\tau^{5} 
		\end{pmatrix}$
		\\	\hspace*{0.5cm}$k=7, l=0 \hspace*{1cm} k=6, l=4 \hspace*{1.2cm} k=5, l=5 \hspace*{0.8cm} k=4, l=3\hspace*{1.4cm} k=3,l=7$
		\item$\begin{pmatrix}
			g^{9}\tau^{13} & g^{9}\tau^{14} \\
			g^{18}\tau^{12}&g^{18}\tau^{13} 
		\end{pmatrix},$
		$\begin{pmatrix}
			g^{9}\tau^{11} & g^{9}\tau^{12} \\
			g^{18}\tau^{10}&g^{18}\tau^{11} 
		\end{pmatrix},$
		$\begin{pmatrix}
			g^{9}\tau^{3} & g^{9}\tau^{4} \\
			g^{18}\tau^{2}&g^{18}\tau^{3} 
		\end{pmatrix}, $
		$\begin{pmatrix}
			g^{9}\tau^{7} & g^{9}\tau^{8} \\
			g^{18}\tau^{6}&g^{18}\tau^{7} 
		\end{pmatrix}$
		\\	\hspace*{0.5cm}$k=8, l=2 \hspace*{1.2cm} k=9, l=1 \hspace*{1.4cm} k=10, l=6 \hspace*{0.8cm} k=11, l=8$
	\end{ex}
	
	\subsection{Properties of Robinson-Schensted correspondence}
	\begin{thm}\label{prop 1}
		If $\tau^{-(i-1)}g^{p^{r-s}}\tau^{*}\tau^{j-1}\tau^{lp^{s-1}(p-1)}\in G_{r}~(\text{and}~SG_{r})$ corresponds to the pair of tableaux $(P_{i},P_{j})$, then $\tau^{-(j-1)}g^{p^{r-s}}\tau^{*}\tau^{i-1}\tau^{lp^{s-1}(p-1)}\in G_{r}~(\text{and}~SG_{r})$ corresponds to the pair of tableaux $(P_{j},P_{i})$, where $1\leq i,j\leq p^{s-1}(p-1),~0\leq l\leq p^{r-s}-1$, and $1\leq s\leq r$ (or $1\leq s\leq r-1$).
	\end{thm} 
	\begin{proof}
		Considering that the integer $i$ is assigned to the tableau $P_{i}$ and the integer $j$ is assigned to the tableau $P_{j}$, then for each fixed $l$, the element $\tau^{-(j-1)}g^{p^{r-s}}\tau^{*}\tau^{i-1}\tau^{lp^{s-1}(p-1)}$ corresponds to the pair of tableaux $(P_{j},P_{i})$.
	\end{proof}
		\begin{Def}	
		Let $a_{1},a_{2}$ be two elements in $G_{r}~(\text{and}~SG_{r})$. Then $a_{1}$ and $a_{2}$ differ by a \textbf{\textit{p}}-Knuth relation, denoted as $a_{1}\overset{\mathfrak{p}}{\cong} a_{2}$ if 
		$$a_{1}=\tau^{-(i-1)}g^{p^{r-s}}\tau^{*}\tau^{j-1}\tau^{lp^{s-1}(p-1)}~\text{and}~ a_{2}=\tau^{-(i-1)}g^{p^{r-s}}\tau^{*}\tau^{k-1}\tau^{lp^{s-1}(p-1)},$$
		where $ 1\leq i,j,k\leq p^{s-1}(p-1)$, $0\leq l\leq p^{r-s}-1$ and $1\leq s \leq r$ (or $1\leq s \leq r-1$ for $SG_{r}$).
	\end{Def}
	\begin{Def}
		Let $b_{1},b_{2}$ be two elements in $G_{r}~(\text{and}~SG_{r})$. Then $b_{1}$ and $b_{2}$ differ by a \textbf{\textit{p}}-dual Knuth relation of the first kind, denoted as $b_{1}\overset{\mathfrak{p}^{*}}{\cong} b_{2}$, if
		$$b_{1}=\tau^{-(i-1)}g^{p^{r-s}}\tau^{*}\tau^{j-1}\tau^{lp^{s-1}(p-1)}~\text{ and}~ b_{2}=\tau^{-(k-1)}g^{p^{r-s}}\tau^{*}\tau^{j-1}\tau^{lp^{s-1}(p-1)},$$ 
		where $1\leq i,j,k\leq p^{s-1}(p-1)$, $0\leq l\leq p^{r-s}-1$ and
		$1\leq s \leq r$ (or $1\leq s \leq r-1$ for $SG_{r}$).
	\end{Def}
	\begin{thm}
		If two elements, $a_{1}$ and $a_{2}$, from $G_{r}~(\text{and}~SG_{r})$ are \textbf{\textit{p}}-Knuth equivalent (of both the first and second kind), then their corresponding $P_{i}$-tableaux are the same.
	\end{thm}
	\begin{proof}
		Note that for each fixed $l$, the integer $i$ in both $a_{1}$ and $a_{2}$ corresponds to the $P_{i}$ tableau, and it remains constant in both elements. Clearly, the proof follows.
	\end{proof}
	\begin{thm}
		If two elements, $b_{1}$ and $b_{2}$, from $G_{r}~(\text{and}~SG_{r})$ are \textbf{\textit{p}}-dual Knuth equivalent (of both the first and second kind), then their corresponding $P_{j}$ tableaux are the same.
	\end{thm}
	\begin{proof}
		Note that for each fixed $l$, the integer $j$ in both $b_{1}$ and $b_{2}$ corresponds to the $P_{j}$ tableau, and it remains constant in both elements. Clearly, the proof follows.
	\end{proof}
	\section{Lacunary Cauchy identity}
	In this section, we discuss the merging of blocks in a Young diagram. We begin by defining the concept of \textbf{\textit{p}}-content and \textbf{\textit{p}}-Schur function and subsequently introduce the lacunary Cauchy identity. We extended the Cauchy identity to our \textbf{\textit{p}}-Brattelli diagram, similar to the classical case of the Robinson-Schensted-Knuth algorithm, resulting in the Cauchy identity. 
	\begin{lem}\label{merging}
		Let $\lambda(kp-(k+1),i)\in \Lambda$ be a hook partition of size $kp-(k+1)$, where $k\geq2$. Here $i=\frac{kp-(k+1)-1}{2}$ when $k$ is even and $\frac{(k-1)p-(k+1)}{2}\leq i < \frac{k+1}{2}(p-1)$ when $k$ is odd. Then the total number of ways in which a \textbf{\textit{p}}-Young diagram with $k$ blocks merged into $l^{'}$ blocks is $2^{k-1}$, where $1\leq l^{'}\leq l$ and $0\leq l\leq k-2$.  
	\end{lem}
	\begin{proof}
		To construct the \textbf{\textit{p}}-Young diagram of the shape $\lambda(kp-(k+1),i)$, blocks of size $p-1$ and $p-2$ are utilized. The construction process to achieve the shape $\lambda(kp-(k+1),i)$ starts with the $p-1$ block, followed by the addition of the $p-2$ block, and subsequently continues with the addition of the $p-1$ block.
		
		The number of blocks decreases to $k-1$ if the first two blocks of sizes $p-1$ and $p-2$ are combined into a single block. These $k-1$ blocks can be merged as follows: $(p(k-l)-(k+1-l),a_{1}(p-1),a_{2}(p-1),\dots,a_{l^{'}}(p-1))$ ($i^{th}$ coordinate indicates the $i^{th}$ block), where $\sum\limits_{i=1}^{l^{'}} a_{i} = l,~a_{i}\geq 1$, $1\leq l^{'} \leq l$ and $1\leq l\leq k-2$. In other words,  $\sum\limits_{i=1}^{l^{'}} a_{i} = l$ forms a composition of $l$. There are $2^{l-1}$ different types of \textbf{\textit{p}}-Young diagrams for each $l$. In total, the number of such $p$-Young diagrams is $1+\sum\limits_{l=1}^{k-2}2^{l-1}=2^{k-2}$.
		There is only one \textbf{\textit{p}}-Young diagram with one block when $l=0$.
		
		If the first two blocks are not combined, then the first block remains fixed with the size $p-1$, and the remaining $k-1$ blocks can be merged as $(p-1,p(k-1-l)-(k-l),a_{1}(p-1),a_{2}(p-1),\dots,a_{l^{'}}(p-1))$($i^{th}$ co-ordinate indicates the $i^{th}$ block), where $\sum\limits_{i=1}^{l^{'}} a_{i} = l,~a_{i}\geq 1$, $1\leq l^{'} \leq l$, and $1\leq l\leq k-2$. In other words,  $\sum\limits_{i=1}^{l^{'}} a_{i} = l$ forms a composition of $l$. There are $2^{l-1}$ different types of \textbf{\textit{p}}-Young diagrams for each $l$. These \textbf{\textit{p}}-Young diagrams can be found in total: $1+\sum\limits_{l=1}^{k-2}2^{l-1}=2^{k-2}$.
		We have only one \textbf{\textit{p}}-Young diagram with $2$ blocks when $l=0$.
		
		Therefore, the total number of \textbf{\textit{p}}- Young diagram of shape $\lambda(kp-(k+1),i)$ is $2^{k-1}$.
	\end{proof}
	\begin{Def}
		Let $T^{l,l^{'},J}$ be a semi-standard \textbf{\textit{p}}-Young tableau $\left(\textbf{\textit{SSYT}}_{\textbf{\textit{p}}}\right)$ of shape $\lambda(p^{r-s}[2sp-(2s+1)],i)$ consisting of either $l^{'}+1$ or $l^{'}+2$ blocks of size $p^{r-s}(p(2s-l)-(2s+1-l))+a_{1}p^{r-s}(p-1)+a_{2}p^{r-s}(p-1)+\dots+a_{l^{'}}p^{r-s}(p-1)$ or $p^{r-s}(p-1)+p^{r-s} (p(2s-1-l)-(2s-l))+a_{1}p^{r-s}(p-1)+a_{2}p^{r-s}(p-1)+\dots+a_{l^{'}}p^{r-s}(p-1)$, respectively, where the blocks are filled with entries $(j_{1},j_{2},\dots,j_{l^{'}+1}) =J$, and these entries are strictly increasing along the first row and strictly increasing along the first column. 
		
		We define the monomials as $$x^{\alpha(T^{l,l^{'},J})}:=x^{p^{r-s}(p(2s-l)-(2s+1-l))}_{j_{1}}x^{a_{1}p^{r-s}(p-1)}_{j_{2}}x^{a_{2}p^{r-s}(p-1)}_{j_{3}}\cdots x^{a_{l^{'}}p^{r-s}(p-1)}_{j_{l^{'}+1}}$$ and $$x^{\alpha(T^{l,l^{'},J})}:=x^{p^{r-s}(p-1)}_{j_{1}}x^{p^{r-s}(p(2s-l-1)-(2s-l))}_{j_{2}}x^{a_{1}p^{r-s}(p-1)}_{j_{3}}x^{a_{2}p^{r-s}(p-1)}_{j_{4}}\cdots x^{a_{l^{'}}p^{r-s}(p-1)}_{j_{l^{'}+2}}$$
		where $\sum\limits_{i=1}^{l^{'}}a_{i}=l,~a_{i}\geq 1$ $0\leq l \leq k-2$.
		
		Similarly, we define $\textbf{\textit{SSYT}}_{\textbf{\textit{p}}}$, $T^{l,l^{'},J}$, and monomial for the shape $\lambda(p^{r-s}[(2s+1)p-2(s+1)],i)$.
	\end{Def}
	
	\noindent Let us define the set,
	
	\noindent $\mathbb{Q}^{(r,s)}_{2s}=$
	\begin{align*}
		& \left\{\alpha^{(1)}_{l,l^{'},J}/ 0\leq l\leq 2s-2,~1\leq l^{'}\leq l,~J=\left\{j_{1},j_{2},\dots, j_{l^{'}+1}\right\}\subseteq\left\{i_{1},i_{2},\dots,i_{2s}\right\}\subset\mathbb{N} \right\} \\ & \cup \left\{\alpha^{(2)}_{l,l^{'},J}/ 0\leq l\leq 2s-2,~1\leq l^{'}\leq l,~J=\left\{j_{1},j_{2},\dots, j_{l^{'}+2}\right\}\subseteq\left\{i_{1},i_{2},\dots,i_{2s}\right\}\subset\mathbb{N} \right\}
	\end{align*}
	\noindent $\mathbb{Q}^{(r,s)}_{2s+1}=$
	\begin{align*}
		&  \left\{\gamma^{(1)}_{l,l^{'},J}/ 0\leq l\leq 2s-1,~1\leq l^{'}\leq l,~J=\left\{j_{1},j_{2},\dots, j_{l^{'}+1}\right\}\subseteq\left\{i_{1},\dots,i_{2s+1}\right\}\subset\mathbb{N} \right\} \\ & \cup \left\{\gamma^{(2)}_{l,l^{'},J}/ 0\leq l\leq 2s-1,~1\leq l^{'}\leq l,~J=\left\{j_{1},j_{2},\dots, j_{l^{'}+2}\right\}\subseteq\left\{i_{1},\dots,i_{2s+1}\right\}\subset\mathbb{N} \right\} 	
	\end{align*}
	where $\alpha^{(1)}_{l,l^{'},J}(\text{and}~\gamma^{(1)}_{l,l^{'},J})$ denotes the $l^{'}+1$-tuple whose $i^{th}$ entry represents the number of times $j_{i}$ occurs in the $\textbf{\textit{SSYT}}_{\textbf{\textit{p}}}$ of shape $\lambda(p^{r-s}[2sp-(2s+1)],i)$ (and $\lambda(p^{r-s}[(2s+1)p-2(s+1)],j)$, respectively), and
	
	$\alpha^{(2)}_{l,l^{'},J}(\text{and}~\gamma^{(2)}_{l,l^{'},J})$ denotes the $l^{'}+2$-tuple whose $i^{th}$ entry represents the number of times $j_{i}$ occurs in the $\textbf{\textit{SSYT}}_{\textbf{\textit{p}}}$ of shape $\lambda(p^{r-s}[2sp-(2s+1)],i)$ (and $\lambda(p^{r-s}[(2s+1)p-2(s+1)],j)$, respectively).
	\begin{Def}
		Each element $\alpha^{(k)}_{l,l^{'},J}\in\mathbb{Q}^{(r,s)}_{2s}(\text{and}~\gamma^{(k)}_{l,l^{'},J}\in\mathbb{Q}^{(r,s)}_{2s+1})$, for $k=1,2$, is referred to as the \textbf{\textit{p}}-content of a $\textbf{\textit{SSYT}}_{\textbf{\textit{p}}}$ of shape $\lambda(p^{r-s}[2sp-(2s+1)],i)$ (and $\lambda(p^{r-s}[(2s+1)p-2(s+1)],j)$, respectively). For simplicity, we denote each element in $\mathbb{Q}^{(r,s)}_{2s}(\text{and}~\mathbb{Q}^{(r,s)}_{2s+1})$ as $\alpha(\text{and}~\gamma)$.
	\end{Def}
	\begin{rem}
		We must distinguish our definition of content from the classical definition. 
		
		We make the following convention: 
		
		The set $\{\lambda(p^{r-1}(p-1),i),~0\leq i <p^{r-1}(p-1) \}$ serves as the indexing set for the one-dimensional representation of $G_{r}$ for all $r\geq1$. The corresponding \textbf{\textit{p}}-Young diagram of the hook partition $\lambda(p^{r-1}(p-1),i),~0\leq i <p^{r-1}(p-1)$, consists of a single block, and the $\textbf{\textit{SSYT}}_{\textbf{\textit{p}}}$ is filled with a single entry, say $n,~n\in \mathbb{N}$.
	\end{rem}
	\begin{rem}\label{nob}	
		\begin{itemize}
			\item[1.] Let $\lambda(p^{r-s}(2sp-(2s+1)),k),~p^{r-s}[sp-(s+1)]\leq k<sp^{r-s}(p-1)$ represent the hook partitions, which constitute the complete set of irreducible representations of degree $p^{s-1}(p-1)$ for the group $G_{r}$. Furthermore, let $\lambda(2sp-(2s+1),k^{'})$ with $sp-(s+1)\leq k^{'}<s(p-1)$ represent the irreducible representation of degree $p^{s-1}(p-1)$ for the group $G_{s}$.
			
			The $p$-Young diagrams corresponding to $\lambda(2sp-(2s+1),k^{'})$ and $\lambda(p^{r-s}(2sp-(2s+1)),k)$ consist of $2s$ blocks, where the block sizes are arranged respectively as $(p-1,p-2,\underbrace{p-1\dots,p-1}_{2s-2~times})$ and $(p^{r-s}(p-1),p^{r-s}(p-2),\underbrace{p^{r-s}(p-1)\dots,p^{r-s}(p-1)}_{2s-2~times})$.
			\item[2.] Let $\lambda(p^{r-s-1}[(2s+1)p-2(s+1)],k),~p^{r-s-1}[sp-(s+1)]\leq k<(s+1)p^{r-s-1}(p-1)$ represent the hook partitions that form the complete set of irreducible representations of degree $p^{s-1}(p-1)$ for the group $SG_{r}$. Furthermore, let $\lambda((2s+1)p,2(s+1),k^{'})~sp-(s+1)\leq k^{'}<(s+1)(p-1)$ represent the irreducible representation of degree $p^{s-1}(p-1)$ for the group $SG_{s}$. The $p$-Young diagrams corresponding to $\lambda([(2s+1)p-2(s+1)],k^{'})$ and $\lambda(p^{r-s}[(2s+1)p-2(s+1)],k)$ consist of $2s+1$ blocks, where the block sizes are arranged respectively as $(p-1,p-2,\underbrace{p-1\dots,p-1}_{2s-1~times})$ and $(p^{r-s-1}(p-1),p^{r-s-1}(p-2),\underbrace{p^{r-s-1}(p-1)\dots,p^{r-s-1}(p-1)}_{2s-1~times})$.
			
			\item[3.] Each hook partition of shape $\lambda(p^{r-s}(2sp-(2s+1)),k)$ is obtained from the hook partition of shape $\lambda(p^{r-s}(2p-3),k^{'})$ by adding $2(s-1)$ blocks of size $p^{r-s}(p-1)$. Similarly, the hook partition of shape $\lambda(2sp-(2s+1),sp-(s+1))$ is obtained from the hook partition of shape $\lambda(2p-3,p-2)$ by adding blocks $2(s-1)$ of size $p-1$.
			
			The hook partition $\lambda(p^{r-s}(2p-3),i)~(\text{and}~\lambda(2p-3,p-2))$ is obtained from the hook partition of shape $\lambda(p^{r-s}(p-1),j)~(\text{and}~\lambda(p-1,j^{'}))$ of dimension one by adding blocks of size $p^{r-s}(p-2)$ $\textbf{(}\text{and}~p-2\textbf{)}$ respectively.
		\end{itemize}				
	\end{rem}				
	
	\begin{lem}\label{lem mp}
		For $r\geq 1$ and $1\leq s\leq r$,  $$x^{\alpha^{'}}_{\lambda(p^{r-s}[2sp-(2s+1)],i)} = \left(x^{\alpha}_{\lambda(2sp-(2s +1),j)}\right)^{p^{r-s}},$$ where $\alpha=\alpha^{(1)}_{l,l^{'},J}=((2s-l)p-(2s+1-l),a_{1}(p-1),a_{2}(p-1),\dots,a_{l^{'}}(p-1))$, $\alpha^{'}=p^{r-s}\alpha$, 
		$p^{r-s}[sp-(s+1)]\leq i <sp^{r-s}(p-1)$ and $[sp-(s+1)]\leq j<s(p-1)$. 
	\end{lem}
	\begin{proof}
		Fix $i_{1}<i_{2}<\dots<i_{k}$ and choose $\{j_{1},j_{2},\dots,j_{l^{'}+1}\}\subseteq\{i_{1},i_{2},\dots,i_{k}\}$ such that $j_{1}<j_{2}<\dots<j_{l^{'}+1}$.						
		\begin{align*}
			x^{\alpha^{'}}_{\lambda(p^{r-s}[2sp-(2s+1)],i)} & = x^{p^{r-s}[(2s-l)p-(2s+1-l)]}_{j_{1}}x^{a_{1}p^{r-s}(p-1)}_{j_{2}}x^{a_{2}p^{r-s}(p-1)}_{j_{3}}\cdots x^{a_{L}p^{r-s}(p-1)}_{j_{l^{'}+1}} \\ & =\left( x^{(2s-l)p-(2s+1-l)}_{j_{1}}x^{a_{1}(p-1)}_{j_{2}}x^{a_{2}(p-1)}_{j_{3}}\cdots x^{a_{L}(p-1)}_{j_{l^{'}+1}}\right)^{p^{r-s}}
			\\ & = \left(x^{\alpha}_{\lambda(2sp-(2s+1),j)}\right)^{p^{r-s}} 	
		\end{align*}
		
		Similarly, we can prove for $x^{\alpha^{'}}_{\lambda(p^{r-s}[2sp-(2s+1)],i)} = \left(x^{\alpha}_{\lambda(2sp-(2s +1),j)}\right)^{p^{r-s}}$ where $\alpha=\alpha^{(2)}_{l,l^{'},J}$ and $\alpha^{'}=p^{r-s}\alpha$.  
	\end{proof}
	
	\begin{lem}\label{lem mp 1}
		For $r\geq 2$ and $1\leq s\leq r$,  $$x^{\gamma^{'}}_{\lambda(p^{r-s-1}[(2s+1)p-2(s+1)],i)} = \left(x^{\gamma}_{\lambda((2s+1)p-2(s+1),j)}\right)^{p^{r-s-1}},$$ where $\gamma = \gamma^{(k)}_{l,l^{'},J}$ for $k=1,2$, $\gamma^{'}=p^{r-s-1}\gamma$, $p^{r-s-1}[sp-(s+1)]\leq i <(s+1)p^{r-s-1}(p-1)$ and $[sp-(s+1)]\leq j<(s+1)(p-1)$. 
	\end{lem}
	\begin{proof}
		This proof is similar to Lemma \ref{lem mp}
	\end{proof}
	
	\begin{lem}
		Let $\lambda(p^{r-s}[2sp-(2s+1)],k)\in \Lambda_{1}, ~p^{r-s}[sp-(s+1)]\leq k<sp^{r-s}(p-1) $ be a hook partition of size $p^{r-s}[2sp-(2s+1)]$, where $r\geq1$ and $1\leq s \leq r$. Then the total number of ways in which the \textbf{\textit{p}}-Young diagram with $2s$ blocks can be merged into $l^{'}$ blocks is $2^{2s-1}$, where $1\leq l^{'}\leq l$ and $0\leq l\leq 2s-2$.  
	\end{lem}
	\begin{proof}
		The proof follows from Lemma \ref{merging} and Lemma \ref{lem mp}.
	\end{proof}
	
	\begin{lem}
		Let $\lambda(p^{r-s-1}[(2s+1)p-(2s+2)],k)\in \Lambda_{2}, ~p^{r-s-1}[sp-(s+1)]\leq k<(s+1)p^{r-s-1}(p-1) $ be a hook partition of size $p^{r-s-1}[(2s+1)p-(2s+2)]$, where $r\geq2$ and $1\leq s \leq r-1$. Then the total number of ways in which \textbf{\textit{p}}-Young diagram with $2s+1$ blocks can be merged into $l^{'}$ blocks is $2^{2s}$, where $1\leq l^{'}\leq l$ and $0\leq l\leq 2s-1$.  
	\end{lem}
	\begin{proof}
		The proof follows from Lemma \ref{merging} and Lemma \ref{lem mp 1}.
	\end{proof}
	
	\begin{Def}
		Let $K_{\lambda(p^{r-s}(2s)),\alpha}~(\text{and}~K_{\lambda(p^{r-s-1}(2s+1)),\gamma})$ denote the number of $\textbf{\textit{SSYT}}_{\textbf{\textit{p}}}$ of size $p^{r-s}[2sp-(2s+1)]~(\text{and}~p^{r-s-1}[(2s+1)p-2(s+1)])$ with \textbf{\textit{p}}-content $\alpha\in\mathbb{Q}^{(r,s)}_{2s}~(\text{and}~\gamma \in\mathbb{Q}^{(r,s)}_{2s+1})$. These numbers are called \textbf{\textit{p}}-Kostka numbers.   
	\end{Def}
	\begin{lem}\label{kostka}
		\begin{enumerate}
			\item $K_{\lambda(p^{r-s}(2s)),\beta^{(k)}_{l,l^{'},J}}=   K_{\lambda(2s),\alpha^{(k)}_{l,l^{'},J}}$, where $\beta^{(k)}_{l,l^{'},J}=p^{r-s}\alpha^{(k)}_{l,l^{'},J}$ for $k=1,2$. 
			\item 	$K_{\lambda(p^{r-s}(2s+1)),\beta^{(k)}_{l,l^{'},J}}=   K_{\lambda(2s+1),\gamma^{(k)}_{l,l^{'},J}}$, where $\beta^{(k)}_{l,l^{'},J}=p^{r-s-1}\gamma^{(k)}_{l,l^{'},J}$ for $k=1,2$. 
		\end{enumerate}
	\end{lem}
	\begin{proof}
		\begin{enumerate}
			\item	Since the maximum number of blocks in the $\textbf{\textit{SSYT}}_{\textbf{\textit{p}}}$ of shapes $\lambda([2sp-(2s+1)],k^{'})$ and $\lambda(p^{r-s}[2sp-(2s+1)],k)$ are the same (by remark \ref{nob}), it follows that the number of different types of filling to the $\textbf{\textit{SSYT}}_{\textbf{\textit{p}}}$ with the same number of blocks is also the same.\label{1}
			\item The proof is similar to (\ref{1}).
		\end{enumerate}
	\end{proof}
	
	\begin{Def}
		\noindent\textbf{\textbf{\textit{p}}-Schur function}
		\begin{itemize}
			\item For any partition $\lambda(p^{r-1}(p-1),k)\in\Lambda$, we define a \textbf{\textit{p}}-Schur function as
			$$\mathcal{S}_{\lambda(p^{r-1}(p-1)),k}(x) = \sum_{i\geq1} x^{p^{r-1}(p-1)}_{i}.$$        
			\item For any partition $\lambda(p^{r-s}[2sp-(2s+1)],k) \in \Lambda_{1}$, we define a \textbf{\textit{p}}-Schur function as
			
			$$\mathcal{S}_{\lambda(p^{r-s}[2sp-(2s+1)],k)}(x) =\sum_{\substack{I \\ |I|=2s}} \sum_{\alpha\in \mathbb{Q}^{(r,s)}_{2s}} x^{\alpha(T^{l,l^{'},J})}.$$    
			\item For any partition $\lambda(p^{r-s-1}[(2s+1)p-2(s+1)],k) \in \Lambda_{2}$, we define a \textbf{\textit{p}}-Schur function as
			
			$$\mathcal{S}_{\lambda(p^{r-s-1}[(2s+1)p-2(s+1)],k)}(x) = \sum_{\substack{I\\|I|=2s+1}}\sum_{\gamma\in \mathbb{Q}^{(r,s)}_{2s+1}} x^{\gamma(T^{l,l^{'},J})}.$$                                      \end{itemize}              	
	\end{Def}
	\begin{lem}\label{schur}
		\begin{enumerate}
			\item $S_{\lambda(p^{r-s}[2sp-(2s+1)],k)}(x)=S_{\lambda(p^{r-s}[2sp-(2s+1)],k^{'})}(x)$, where $p^{r-s}[sp-(s+1)\leq k,k^{'}<sp^{r-s}(p-1)]$
			\item $S_{\lambda(p^{r-s-1}[(2s+1)p-2(s+1)],k)}(x)=S_{\lambda(p^{r-s-1}[(2s+1)p-2(s+1)],k^{'})}(x)$, where $p^{r-s}[sp-(s+1)\leq k,k^{'}<(s+1)p^{r-s}(p-1)]$ 
		\end{enumerate}
	\end{lem}		
	\begin{proof}
		\begin{enumerate}
			
			\item The proof is obvious, for each $\textbf{\textit{SSYT}}_{\textbf{\textit{p}}}$ of shape $\lambda(p^{r-s}[2sp-(2s+1)],j)$ and size $p^{r-s}[2sp-(2s+1)]$ is attained by adding different types of block with same size. Since the number of blocks and their sizes are the same for all hook partitions with same size, the filling are also similar, which proves the equality. \label{11}
			
			\item The proof is similar to (\ref{11}). 
		\end{enumerate}
	\end{proof}	
	\begin{thm}(\textbf{Lacunary Cauchy identity})
	\begin{align*}
		\sum_{\lambda\in \Lambda}\mathcal{S}_{\lambda}(x)\mathcal{S}_{\lambda}(y)&= (p-1)^{2}\sum_{i,j\geq 1}g(x^{p-1}_{i}y^{p-1}_{j})+\sum_{s\geq 1}\sum_{\alpha,\beta\in \mathbb{Q}_{\lambda,2s}}K_{\lambda(2s),\alpha}K_{\lambda(2s),\beta}\cdot g(x^{\alpha},y^{\beta})\\&~~~+p^{2} \sum_{s\geq1}\sum_{\gamma,\delta\in\mathbb{Q}_{\lambda,2s+1}}K_{\lambda(2s+1),\gamma}K_{\lambda(2s+1),\delta}\cdot g(x^{\gamma},y^{\delta})
	\end{align*}	
	where $g(x)=\sum_{n\geq 0}p^{2n}x^{p^{n}}$.
\end{thm}
\begin{proof}
	\begin{itemize}
		\item Let $\lambda(p^{r-1}(p-1),k)\in V^{r}_{0}(W^{r}_{0}), ~r\geq 1$ be a hook partition of degree $1$. Then
		\begin{align*}
			\sum_{r\geq 1} \mathcal{S}_{\lambda(p^{r-1}(p-1),k)}(x)\mathcal{S}_{\lambda(p^{r-1}(p-1),k)}(y)
			&= \sum_{i,j\geq 1}(p-1)^{2}x^{p-1}_{i}y^{p-1}_{j}+p^{2}(p-1)^{2}(x^{p-1}_{i})^{p}(y^{p-1}_{j})^{p}\\&+p^{4}(p-1)^{2}(x^{p-1}_{i})^{p^{2}}(y^{p-1}_{j})^{p^{2}}+\dots \\
			&=\sum_{i,j\geq 1}(p-1)^{2}\sum_{n\geq0}p^{2n}(x^{p-1}_{i})^{p^{n}}(y^{p-1}_{j})^{p^{n}}\\
			& = (p-1)^{2}\sum_{i,j\geq 1}g(x^{p-1}_{i}y^{p-1}_{j})
		\end{align*}
		\item Let $\lambda(p^{r-s}[2sp-(2s+1)],k)\in V^{r}_{s}, ~r\geq 1$ be a hook partition of degree $p^{s-1}(p-1)$. Then
		\begin{align*}
			&\hspace*{-5cm}\sum_{\substack{1\leq s\leq r, \\r\geq 1}} \mathcal{S}_{\lambda(p^{r-s}[2sp-(2s+1)],k)}(x)\mathcal{S}_{\lambda(p^{r-s}[2sp-(2s+1)],k^{'})}(y)\\
			&\hspace*{-3.3cm}= \sum_{\substack{1\leq s\leq r, \\r\geq 1}}\left(\sum_{\alpha\in\mathbb{Q}^{(r,s)}_{2s}}x^{\alpha}\sum_{\beta\in\mathbb{Q}^{(r,s)}_{2s}}y^{\beta}\right)
		\end{align*}
		By using Lemma \ref{kostka}, Lemma \ref{schur}, Lemma \ref{lem mp} and rearranging the summation, we get
		\begin{align*}
			&\hspace*{1cm}=\sum_{s\geq 1}\sum_{\alpha,\beta\in\mathbb{Q}^{(s,s)}_{2s}}K_{\lambda(2s),\alpha}K_{\lambda(2s),\beta}\cdot(x^{\alpha}y^{\beta})+K_{\lambda(2s),\alpha}K_{\lambda(2s),\beta}\cdot p^{2}(x^{\alpha}y^{\beta})^{p}+\dots \\
			&\hspace*{1cm}=\sum_{s\geq 1}\sum_{\alpha,\beta\in\mathbb{Q}^{(s,s)}_{2s}}K_{\lambda(2s),\alpha}K_{\lambda(2s),\beta}\cdot g(x^{\alpha}y^{\beta})
		\end{align*}
		\item Let $\lambda(p^{r-s-1}[(2s+1)p-2(s+1)],k)\in W^{r}_{s}, ~r\geq 2$ be a hook partition of degree $p^{s-1}(p-1)$. Then
		\begin{align*}
			&\hspace*{-3.3cm}\sum_{\substack{1\leq s\leq r-1, \\r-1\geq 1}} \mathcal{S}_{\lambda(p^{r-s-1}[(2s+1)p-2(s+1)],k)}(x)\mathcal{S}_{\lambda(p^{r-s-1}[(2s+1)p-2(s+1)],k^{'})}(y)\\
			&\hspace*{-2cm}=\sum_{\substack{1\leq s\leq r-1, \\r-1\geq 1}}\left(\sum_{\gamma\in\mathbb{Q}^{(r,s)}_{2s+1}}x^{\gamma}\sum_{\delta \in\mathbb{Q}^{(r,s)}_{2s+1}}y^{\delta}\right)
		\end{align*}
		By using Lemma \ref{kostka}, Lemma \ref{schur}, Lemma \ref{lem mp 1} and  rearranging the summation, we get
		\begin{align*}
			&\quad= \sum_{s\geq 1}\sum_{\gamma,\delta\in\mathbb{Q}^{(s,s)}_{2s+1}}p^{2}K_{\lambda(2s+1),\gamma}K_{\lambda(2s+1),\delta}\cdot (x^{\gamma}y^{\delta})+p^{2}K_{\lambda(2s+1),\gamma}K_{\lambda(2s+1),\delta}\cdot p^{2}(x^{\gamma}y^{\delta})^{p}+\dots \\
			&\quad=p^{2}\sum_{s\geq 1}\sum_{\gamma,\delta\in\mathbb{Q}^{(s,s)}_{2s+1}}K_{\lambda(2s+1),\gamma}K_{\lambda(2s+1),\delta}\cdot g(x^{\gamma}y^{\delta})
		\end{align*}
	\end{itemize}
\end{proof}

	\begin{rem}
		\textbf{The series $g(x)=\sum\limits_{n\geq 0}p^{2n}x^{p^{n}}$ obtained in the above theorem is a ``lacunary series". The series $g(x)$ converges absolutely, but we could not find the closed form expression of it.}
	\end{rem}
	\section{Conclusions}
	This paper establishes a Robinson-Schensted correspondence between the group $G_{r}$ (and $SG_{r}$) and the pairs of  standard \textit{p}-Young tableaux. Our approach, which utilizes matrix units computed from the orthogonal primitive idempotents, provides a fresh perspective on the representation theory of these groups. In a future paper, we will build upon this work to define the descents for each tableau in sets $\Lambda_{1}$ and $\Lambda_{2}$. We then derive the generating function for the total number of descents at each vertex. In addition, we present several identities that establish connections between the tableaux and irreducible representations.


\bigskip
	
	\begin{thebibliography}{9}
		\bibitem{[PT1]}  M. Parvathi., A. Tamilselvi., Robinson-Schensted Correspondence for the Signed Brauer Algebras, The Electronic Journal of Combinatorics., 14 (2007). 		\bibitem{[PT2]} M. Parvathi., A. Tamilselvi.,  The Robinson-Schensted correspondence for the $G$-Brauer algebras, Contemporary Mathematics Proceedings., 456 (2008).
		\bibitem{[R]}  G. de B. Robinson., On the Representations of the Symmetric Group, American Journal of Mathematics., 60 (3) (1938), 745-760. 	
		 		
		\bibitem{[BS]} B. E. Sagan., The Symmetric Group Representations, Combinatorial Algorithms, and Symmetric Functions, Second edition, Springer, 2001.
		
		\bibitem{[CS]} C. Schensted., Longest Increasing and Decreasing Subsequences, Canadian Journal of Mathematics., 13 (1961), 179-191.
		
		
		\bibitem{[RS]} R. P. Stanley., Enumerative Combinatorics, Volume 2, Cambridge University Press, 1999. 
		
			
		
		\bibitem{[TD]} A. Tamilselvi and S. Dhilshath.,  Harmony of Walled Klein-4 Brauer Diagrams and Klein-4 Vacillating Tableaux, Indian Journal of Science and Technology., 16 (48) (2023).
		
		\bibitem{[TH]} A. Tamilselvi and D. Hepsi., Study of Matrix units of the group algebras $KG_{r}$ and $KSG_{r}$, Indian Journal of Science and Technology., 17 (48) (2024), 5119-5129.
		
		
		\bibitem{[TVK]} A. Tamilselvi, A. Vidhya and B. Kethesan., Robinson-Schensted correspondence for the Walled Brauer Algebras and the Walled Signed Brauer Algebras, Algebra and its Applications., Springer Singapore, (2016), 195-223.
		
		\bibitem{[HL]} Tom Halverson and Tim Lewandowski., , RSK Insertion for Set Partitions and Diagram Algebras, The Electronic Journal of Combinatorics., 11 (2) (2005).
		
		\bibitem{[AVZ]} A. V. Zelevinsky., A Generalization of the Littlewood-Richardson Rule and the Robinson-Schensted-Knuth Correspondence, J. Algebra., 69, (1981), 82-94.
	\end{thebibliography}
\end{document}